\documentclass[titlepage,leqno,11pt]{amsart}

\usepackage{graphicx}
\usepackage{psfrag}
\usepackage{latexsym}
\usepackage{amssymb}
\usepackage{amsfonts}
\usepackage{amsmath}
\usepackage{amsthm}
\usepackage{color}
\usepackage[margin=1in]{geometry}

\newcommand\R{{\mathbb R}}
\newcommand\V{{\mathbb V}}

\newcommand{\E}{\mathbb{E}}

\newtheorem{theorem}{Theorem}[section]

\newtheorem{corollary}[theorem]{Corollary}
\newtheorem{lemma}[theorem]{Lemma}
\newtheorem{proposition}[theorem]{Proposition}

\theoremstyle{definition}
\newtheorem{definition}[theorem]{Definition}

\theoremstyle{remark}
\newtheorem{remark}[theorem]{Remark}
\newtheorem{question}[theorem]{Question}

\numberwithin{equation}{section}

\newcommand{\ep}{\epsilon}
\newcommand{\del}{\delta}

\newcommand{\reals}{\mathbb{R}}
\newcommand{\Heis}{\mathbb{H}}

\newcommand{\nats}{\mathbb{N}}

\newcommand{\nbhd}{\mathcal{N}}

\newcommand{\inv}{^{-1}}

\newcommand{\ovl}{\overline}

\newcommand{\til}{\widetilde}

\newcommand{\Hdim}{\mathcal{H}}
\newcommand{\cF}{{\mathcal{F}}}
\newcommand{\subeq}{\subseteq}
\newcommand{\supeq}{\supseteq}
\newcommand{\bslash}{\backslash}

\newcommand{\mbf}{\mathbf}
\newcommand{\Lloc}{{\rm L}_{\loc}}

\newcommand{\cX}{\mathcal{X}}

\newcommand{\cH}{\Hdim}

\def\loc{\operatorname{loc}}
\def\diam{\operatorname{diam}}
\def\card{\operatorname{card}}
\def\dist{\operatorname{dist}}

\def\Lip{\operatorname{Lip}}

\def\spa{\operatorname{span}}

\def\Xint#1{\mathchoice
{\XXint\displaystyle\textstyle{#1}}%
{\XXint\textstyle\scriptstyle{#1}}%
{\XXint\scriptstyle\scriptscriptstyle{#1}}%
{\XXint\scriptscriptstyle\scriptscriptstyle{#1}}%
\!\int}
\def\XXint#1#2#3{{\setbox0=\hbox{$#1{#2#3}{\int}$}
\vcenter{\hbox{$#2#3$}}\kern-.5\wd0}}

\def\mint{\Xint-}

\author[Z.M. Balogh, J.T. Tyson, and K. Wildrick]{Zolt\'an M. Balogh, Jeremy T. Tyson, Kevin Wildrick}

\address{Z. Balogh: Mathematisches Institut, Universit\"at Bern, Sidlerstrasse 5, 3012 Bern, Switzerland ({\tt balogh.zoltan@math.unibe.ch})}

\address{J.T. Tyson: Department of Mathematics, University of Illinois at Urbana-Champaign, 1409 W Green Street, Urbana, IL 61801, USA ({\tt tyson@math.uiuc.edu})}

\address{K. Wildrick: Mathematisches Institut, Universit\"at Bern, Sidlerstrasse 5, 3012 Bern, Switzerland ({\tt kevin.wildrick@math.unibe.ch})}
\keywords{Sobolev mapping, Ahlfors regularity, Poincar\'e inequality, foliation, David--Semmes regular mapping \\ 2010 \emph{Mathematics subject classification.} Primary: 46E35, 28A78; Secondary: 46E40, 53C17, 30L99}

\thanks{The first and third authors were supported by the Swiss National Science Foundation, European Research Council Project CG-DICE, and the European Science Council Project HCAA. The second author was supported by NSF grants DMS-0901620 and DMS-1201875.}

\begin{document}

\title[Dimension distortion in metric spaces]{Dimension distortion by Sobolev mappings in foliated metric spaces}

\begin{abstract}
We quantify the extent to which a supercritical Sobolev mapping can increase the dimension of subsets of its domain, in the setting of metric measure spaces supporting a Poincar\'e inequality. We show that the set of mappings that distort the dimensions of sets by the maximum possible amount is a prevalent subset of the relevant function space. For foliations of a metric space $X$ defined by a David--Semmes regular mapping $\pi \colon X \to W$, we quantitatively estimate, in terms of Hausdorff dimension in $W$, the size of the set of leaves of the foliation that are mapped onto sets of higher dimension. We discuss key examples of such foliations, including foliations of the Heisenberg group by left and right cosets of horizontal subgroups.   
\end{abstract}

\date{\today}

\maketitle

%\setcounter{tocdepth}{1}
%\tableofcontents

\section{Introduction}

Let $N$ and $N'$ be differentiable manifolds of dimensions $n \geq n'$. For every $y \in N'$, the preimage $\pi\inv(y)$ under a submersion $\pi \colon N \to N'$ is a submanifold of $N$ of dimension $n-n'$. In this way, the map $\pi$ defines a foliation of $N$ parameterized by $N'$. The canonical such submersion is the orthogonal projection of $\reals^n$ onto the codimension one subspace spanned by all but the $i^{\text{th}}$ coordinate vector, $i=1,\hdots,n$. The resulting foliation of $\reals^n$ by parallel straight lines features in the theory of Sobolev mappings. Indeed, a mapping $f\in L^p(\reals^n,\reals^m)$ is in the Sobolev space $W^{1,p}(\reals^n;\reals^m)$ if and only if, up to choice of representative, for each $i=1,\hdots, n$, each coordinate function of $f$ is absolutely continuous on $\Hdim^{n-1}$-almost every line in the above foliation, and the resulting partial derivatives are in $L^{p}(\reals^n)$.

Sobolev mappings between metric spaces are of growing interest and importance in modern analysis, geometric group theory, and geometric measure theory.  While there are several (often equivalent) definitions of such Sobolev maps, in each approach Sobolev mappings are assumed or shown to be absolutely continuous along ``almost every curve".   When the space under consideration is equipped with a foliation by curves that is parameterized by another space $W$, it is natural that ``almost every curve" refer to a measure on $W$, as in the Euclidean case above. In other words, a Sobolev mapping should preserve or decrease the dimension of almost every leaf of a given foliation by curves.

Our first result states that while a Sobolev mapping may substantially increase the dimension of the remaining measure zero set of curves, this increase is controlled. In fact, there are universal bounds on the dimension increase under a supercritical Sobolev mapping. Here and henceforth in this paper we denote by $\cH^\alpha_Y$ the $\alpha$-dimensional Hausdorff measure in a metric space $Y$, and by $\dim_Y A$ the Hausdorff dimension of a subset $A$ of the space $Y$. The assumptions on the metric space $X$ in Theorem \ref{Kaufman metric intro} are standard and explained in Section \ref{notation-section} below, as is the notion of a Sobolev space based on upper gradients.

\begin{theorem}\label{Kaufman metric intro}  Let $X$ be a proper metric measure space that is locally $Q$-homogeneous and supports a local $Q$-Poincar\'e inequality. Let $Y$ be any metric space. For $p>Q$, if $f \colon X \to Y$ is a continuous mapping that has an upper gradient in $\Lloc^p(X)$, then
\begin{equation}\label{Kaufman-estimate-intro}
\dim_Y f(E)\leq \frac{p\dim_X E}{p-Q+\dim_X E}
\end{equation}
for any subset $E \subeq X$. Moreover, if $\Hdim_X^Q(E)=0$, then $\Hdim_Y^Q(f(E))=0$. 
\end{theorem}

%\begin{theorem}\label{Kaufman-intro} Let $\Omega \subeq \reals^n$ be a domain and let $Y$ be a Banach space. Each continuous mapping $f \colon \Omega \to Y$ in ${\rm W}^{1,p}(\Omega;Y)$, $p>n$, satisfies the following universal bound on dimension increase: for any Borel set $E\subset\R^n$ with $\cH^s(E)<\infty$ for some $s<n$, one has $\cH_Y^\alpha(f(E))=0$, where
%\begin{equation}\label{alpha-value}
%\alpha = \frac{p s}{p-n +s}.
%\end{equation}
%Moreover, if $\cH^n(E)=0$, then $\cH_Y^n(f(E))=0$ as well.
%\end{theorem}

 In the setting of quasiconformal mappings between domains in Euclidean space, such dimension distortion estimates have been known for many years; see, for instance, Gehring \cite{Gehring73}, Gehring--V\"ais\"al\"a\ \cite{GehringV73} and Astala \cite{Astala1994}. Theorem \ref{Kaufman metric intro} was also known in the Euclidean Sobolev setting; see for instance Kaufman \cite[Theorem 1]{Kaufman}. 

The main difficulty in proving Theorem \ref{Kaufman metric intro} is that in the general metric setting, the usual analogue of a Euclidean dyadic cube need not be bi-Lipschitz equivalent to a ball of a comparable diameter. Hence, nice essentially disjoint coverings of sets need not exist. We overcome this obstacle by using maximal functions; see Lemma~\ref{shrink}.  In the case that the set $E$ under consideration in Theorem~\ref{Kaufman metric intro} is Ahlfors regular, we reach the stronger conclusion that $f(E)$ has zero measure in the appropriate dimension; see Theorem \ref{Q-regular} below.

We also prove that Theorem \ref{Kaufman metric intro} is sharp whenever the domain is Ahlfors regular, and that the collection of Sobolev mappings which increase the dimension of a given compact set $E$ by the maximum possible amount is \emph{prevalent}, a notion of genericity in Banach spaces (see \cite{Christensen}, \cite{Aronszajn}, \cite{NullCoincide}, \cite{hunt-sauer-yorke1992}, \cite{sauer-yorke1997}). Here the relevant Banach space is a Newtonian--Sobolev space, defined in Section \ref{Morrey section}.

\begin{theorem}\label{universal-sharp}
Let $(X,d,\mu)$ be a metric measure space that is locally Ahlfors $Q$-regular for some $Q>0$. For $0\leq s \leq Q$ and $p>Q$, set
$$\alpha = \frac{ps}{p-Q+s}.$$
Let $E$ be a compact subset of $X$ such that $\Hdim^s(E)>0$.  Then for all $N \in \nats$ greater than $\alpha$, there exists a continuous map $f:X \to \R^N$ such that $f$ has an upper gradient in ${\rm L}^p(X)$ and $\dim f(E) \geq \alpha$. Moreover, the set of such functions forms a prevalent set in the Newtonian--Sobolev space ${\rm N}^{1,p}(X;\reals^N)$.
\end{theorem}

We now turn to the study of dimension increase under Sobolev mappings of generic leaves in parameterized families of subsets. Consider a continuous supercritical Sobolev mapping $f \colon X \to Y$ as in Theorem \ref{Kaufman metric intro} and fix a dimension $0\leq s \leq Q$ and a target dimension $s \leq \alpha\leq ps/(p-Q+s)$. If $X$ is foliated by subsets of dimension no greater than $s$, how many leaves of the foliation can be mapped by $f$ onto sets of dimension at least $\alpha$? The answer will be given in terms of Hausdorff measures on the parameterizing space for the foliation.

The preceding question has been thoroughly studied in Euclidean space. The first two authors together with Monti \cite{BMT} studied supercritical ($p>n$) and borderline ($p=n$) Sobolev mappings on foliations arising from the orthogonal projection of $\reals^n$ onto a subspace of arbitrary dimension. In the same setting, Hencl and Honz\'{i}k \cite{HenclHonzik} considered the sub-critical case ($p<n$). Bishop and Hakobyan \cite{Bishop} have recently addressed finer questions for the behavior of planar quasiconformal mappings along lines.  

In this general metric setting, we must first give precise meaning to the notion of foliation. David and Semmes \cite{RegularBetween} introduced a class of mappings between metric spaces that is analogous to the class of submersions between differentiable manifolds. In this paper, we study foliations arising from local versions of David--Semmes regular mappings.

\begin{definition}\label{regular} Let $s \geq 0$. A surjection  $\pi \colon X \to W$ between proper metric spaces is said to be \emph{locally David--Semmes $s$-regular} (for short, {\it locally $s$-regular}) if for every compact subset $K \subeq X$, $\pi|_K$ is Lipschitz and there is a constant $C \geq 1$ and a radius $r_0>0$ such that for every ball $B \subeq W$ of radius $r<r_0$, the truncated preimage $\pi \inv(B) \cap K$ can be covered by at most $Cr^{-s}$ balls in $X$ of radius $Cr$.
\end{definition}

An easy calculation shows that given a locally $s$-regular mapping $\pi \colon X \to W$, a compact subset $K\subeq X$, and a point $a \in W$,
$$\Hdim^s_X(\pi\inv(a) \cap K) \leq C,$$
where $C< \infty$ depends only on the constant associated to $K$ in Definition~\ref{regular}. In particular, the leaves $\pi\inv(a)$ have Hausdorff dimension no greater than $s$. However, leaves can have Hausdorff dimension strictly less than $s$; this situation occurs naturally for certain foliations of the Heisenberg group, as we will see later.

Let $\pi \colon X \to W$ be a locally $s$-regular mapping. The triple $(X,W,\pi)$ will be called an \emph{$s$-foliation} of $X$. If the value of $s$ is unimportant, we will refer to an $s$-foliation as a \emph{metric foliation}. Given a point $a \in W$, the set $\pi\inv(a)$ is called a \emph{leaf} of the foliation.

\begin{theorem}\label{foliation}
Let $Q\geq 1$ and $0<s<Q$. Let $(X,d_X,\mu)$ be a proper metric measure space that is locally $Q$-homogeneous, supports a local $Q$-Poincar\'e inequality, and is equipped with an $s$-foliation $(X,W,\pi)$. Let $Y$ be any metric space. For $p>Q$, if $f \colon X \to Y$ is a continuous mapping that has an upper gradient in $\Lloc^p(X)$, then
\begin{equation}\label{foliation-estimate}
\dim \{a \in W: \dim(f(\pi\inv(a)))\geq \alpha\} \leq (Q-s)-p\left(1-\frac{s}{\alpha}\right)
\end{equation}
for each $\alpha \in \left( s,\frac{ps}{p-Q+s}\right\rbrack$.
\end{theorem}

The metric setting presents several obstacles to a straightforward adaptation of the Euclidean proof given in \cite[Theorem~1.3]{BMT}. We make heavy use of the machinery of geometric measure theory in metric spaces espoused in \cite{Mattila}.

A motivating example of a metric measure space to which our results apply is the Heisenberg group $\Heis$. Of particular interest are foliations of $\Heis$ by either left or right cosets of a given horizontal subgroup. When the leaves are left cosets, they are also horizontal, and the natural parameterizing space is Euclidean.  However, when the leaves are right cosets, they are rarely horizontal, and the natural parameterizing space is the Grushin plane, a sub-Riemannian metric space homeomorphic but not bi-Lipschitz equivalent to $\reals^2$. In this case, the wide generality of Theorem \ref{foliation} is needed.  These and other examples are discussed in Section~\ref{examples section}.

Theorem \ref{foliation} recovers the Euclidean result in the case of orthogonal projections onto subspaces, and is sharp in that setting \cite[Theorem~1.4]{BMT}.  However, it is interesting to note that while the foliation of the Heisenberg group by left cosets of a horizontal subgroup is a $2$-foliation, the dimension of a leaf is only $1$. This prevents Theorem \ref{foliation} from being sharp, and indicates that the framework of David-Semmes foliations is not appropriate. In the article \cite{DimDistHeis}, we provide an alternate framework, based on the Radon-Nikodym theorem, which is more appropriate. Notably, this alternate framework cannot accommodate foliations that are not parameterized by a Euclidean space, such as the foliation of the Heisenberg group by right cosets of a horizontal subgroup.

We now give an outline for this paper. In Section \ref{notation-section}, we establish notation and recall relevant definitions from the theory of analysis on metric spaces. Section \ref{Morrey section} describes the version of Morrey's inequality, a key tool in our proofs, that is valid in the metric measure space setting. We prove Theorems \ref{Kaufman metric intro} and \ref{universal-sharp}  in Section~\ref{Universal section}. Section~\ref{reg foliations section} contains the proof of Theorem \ref{foliation}.
In Section \ref{examples section} we provide examples of metric foliations and discuss the applications of our results. The final Section \ref{questions section} contains some open questions and problems motivated by this work.

\

\paragraph{\bf Acknowledgements.} Research for this paper was conducted during a visit of the third author to the Department of Mathematics of the University of Illinois at Urbana-Champaign in January 2012, and during a visit of the second author to the Institute of Mathematics at the University of Bern in June 2012. We would like to thank these institutions for their hospitality. 

\section{The metric measure space setting}\label{notation-section}

In a metric space $(X,d)$, we denote the open ball centered at a point $x \in X$ of radius $r>0$ by
$$B_X(x,r)=\{y \in X: d(x,y)<r\}$$
and the corresponding closed ball by
$$\ovl{B}_X(x,r) = \{y \in X: d(x,y) \leq r\}.$$
When there is no danger of confusion, we often write $B(x,r)$ in place of $B_X(x,r)$.  A similar convention will be used for all objects that depend implicitly on the ambient space. For a subset $A$ of $X$ and a number $\ep>0$, we denote the $\ep$-neighborhood of $A$ by
$$\nbhd(A,\ep) = \{x \in X : \dist(A,x) < \ep\}.$$
For an open ball $B=B(x,r)$ and a parameter $\lambda > 0$, we set $\lambda B =B(x,\lambda r)$.

A metric space is \emph{proper} if every closed ball is compact. We will only consider proper metric spaces in this paper.

A \emph{metric measure space} is a triple $(X,d,\mu)$ where $(X,d)$ is a metric space and $\mu$ is a measure on $X$. The measure $\mu$ is assumed to be a Borel measure that gives positive and finite value to any non-empty open set. For $E \subeq X$, we denote by $\mu\lfloor E$ the restriction of $\mu$ to $E$. Given a mapping $f$ from $X$ to some other metric space $Y$, we define the push-forward measure of $\mu$ by $f$ as 
$$f_\sharp\mu(U) = \mu(f\inv(U)),$$
where $U \subeq Y$. 

For $t \geq 0$, the $t$-dimensional Hausdorff measure on a metric space $X$ will be denoted by $\Hdim^t_X$ or simply $\Hdim^t$.  For $\ep>0$, the corresponding pre-measure will be denoted by $\Hdim^{t}_{\ep,X}$ or $\Hdim^t_\ep$, and the corresponding content will be denoted by $\Hdim^t_{\infty,X}$ or $\Hdim^t_{\infty}$. Unless otherwise noted, for $E \subeq X$ we denote by $\dim E$ the Hausdorff dimension of the metric space $(E,d_X)$. We refer to Mattila \cite{Mattila} for more details and information about geometric measure theory in metric spaces.

Let $Q > 0$. We say that the metric measure space $(X,d,\mu)$ is \emph{locally $Q$-homogeneous} if for every compact subset $K \subeq X$, there is a radius $R > 0$ and a constant $C \geq 1$ such that $$\frac{\mu(B(x,r_2))}{r_2^Q} \le C \frac{\mu(B(x,r_1))}{r_1^Q}$$
whenever $B(x,r_1) \subeq B(x,r_2)$ are concentric balls centered in $K$. When the value of $Q$ is unimportant, we say that $(X,d,\mu)$ is \emph{locally homogeneous}.

Any locally $Q$-homogeneous space has Hausdorff dimension at most $Q$. In fact, such spaces have Assouad dimension at most $Q$; note that the Assouad dimension is always greater than or equal to the Hausdorff dimension. We will not make use of Assouad dimension in this paper.

Every locally homogeneous metric measure space $(X,d,\mu)$ is \emph{locally doubling}, which means that for every compact subset $K \subeq X$, there is a radius $R > 0$ and a constant $C \geq 1$ such that
$$
\mu(B(x,2r)) \le C \mu(B(x,r))
$$
whenever $B(x,r)$ is a ball centered in $K$ with $r\le R$.

The local homogeneity condition only provides lower bounds on measure. We will occasionally require upper bounds as well. The metric measure space $(X,d,\mu)$ is \emph{locally Ahlfors $Q$-regular} if for every compact subset $K \subeq X$, there is a radius $R > 0$ and a constant $C \geq 1$ such that
$$\frac{r^Q}{C} \leq \mu(B_r) \leq C r^Q$$
whenever $B_r$ is a ball centered in $K$ of radius $r<R$.

We will often consider conditions on spaces and mappings defined using multiplicative constants. When estimating quantities involving such constants, we use the notation $A \lesssim B$ to mean that there is a constant $C \geq 1$, depending only on certain specified and fixed quantities, such that $A \leq C B$.

\section{Sobolev classes and Morrey's estimate}\label{Morrey section}

Let $p>n$, and let $m \in \nats$.  Each mapping in the supercritical Sobolev space ${\rm W}_{\loc}^{1,p}(\reals^n;\reals^m)$, $p>n$, has a $(1-\frac{n}{p})$-H\"older continuous representative satisfying \emph{Morrey's estimate}
\begin{equation}\label{SobEmbed intro} \diam f(Q) \leq c(n,p) \diam Q \left(\mint_Q |D f|^p \ d\Hdim^n\right)^{\frac{1}{p}},
\end{equation}
for each ball or cube $Q \subeq \reals^n$; see, e.g., \cite{Ziemer}. Here $|Df|$ denotes the norm of the matrix of weak partial derivatives of the coordinate functions of $f$, and $c(n,p)$ is a positive constant depending only on $n$ and $p$. This fact is the sole property of Sobolev mappings needed for the results in this paper. Note that by H\"older's inequality, $${\rm W}_{\loc}^{1,p}(\reals^n;\reals^m) \subeq {\rm W}_{\loc}^{1,q}(\reals^n;\reals^m)$$ for all $1 \leq q < p$. Hence, if $f \in {\rm W}_{\loc}^{1,p}(\reals^n;\reals^m)$, the inequality \eqref{SobEmbed intro} also holds with $p$ replaced by any exponent $q \in (n,p]$.

We wish to state a version of the Morrey inequality in the metric measure space setting. Throughout this section, we assume that $(X,d,\mu)$ is a proper metric measure space and that $(Y,d_Y)$ is an arbitrary metric space.

A robust approach to Sobolev spaces of mappings between metric spaces is based on the concept of an upper gradient \cite{Cheeger}, \cite{Acta}, \cite{Nages}. Let $f \colon X \to Y$ be a continuous map, and let $g \colon Y \to [0,\infty]$ be a Borel function.
The function $g$ is an \emph{upper gradient} of $f$ if for every rectifiable curve $\gamma \colon [0,1] \to X$,
$$d_Y(f(\gamma(0)),f(\gamma(1))) \leq \int_\gamma g \ ds.$$
We consider mappings $f$ which have upper gradients in $\Lloc^p(X)$. Such mappings are absolutely continuous on ``most" rectifiable curves in $X$ \cite[Proposition 3.1]{Nages}. When the target space $Y$ is a Banach space, one can define a Banach space of Newtonian-Sobolev mappings ${\rm N}^{1,p}(X;Y)$ consisting of equivalence classes of (not necessarily continuous) mappings in ${\rm L}^p(X;Y)$ with an upper gradient in ${\rm L}^p(X)$. For more details, see \cite{Nages} or \cite{HKST}.

If there are no rectifiable curves in $X$, then any mapping $f \colon X \to Y$ has the zero function as an upper gradient, and so there is no hope for a Morrey estimate.  The $Q$-Poincar\'e inequality remedies this \cite{SobMet}, \cite{Acta}, \cite{LAMS}.

\begin{definition}
Let $p\ge 1$. A metric measure space $(X,d,\mu)$ satisfies a \emph{local $p$-Poincar\'e inequality} if for every compact subset $K \subeq  X$, there are constants $C \geq 1$, $\sigma \geq 1$, and $R>0$ such that if $f \colon X \to \reals$ is a continuous function and $g \colon X \to [0,\infty]$ is an upper gradient of $f$, then
$$\mint_{B}|f-f_B| \ d\mu \leq C \diam B \left(\mint_{\sigma B} g^p \ d\mu\right)^{1/p}$$
for each open ball $B \subeq X$ centered in $K$ of radius less than $R$.
\end{definition}
%
%Although the Poincar\'e inequality stated here is a condition on real-valued measurable functions, it implies a similar condition on continuous mappings from $X$ to $Y$ \cite[Section 4]{HKST}. This is enough to guarantee a Morrey estimate \cite[Theorem 5.1]{SobMet}, \cite[Theorem 6.2]{HKST}.

\begin{theorem} [Haj{\l}asz-Koskela, Heinonen et al.]\label{PoinGivesSob}
Assume that $X$ is locally $Q$-homogeneous, $Q \geq 1$, and supports a local $Q$-Poincar\'e inequality. If $f \colon X \to Y$ is a continuous mapping with an upper gradient $g \in \Lloc^p(X)$ for some $p>Q$, then for each compact subset $K \subeq X$ there exist constants $C \geq 1$, $\sigma \geq 1$, and $R>0$ satisfying
$$\diam f(B) \leq C (\diam B) \left(\mint_{\sigma B} g^p \ d\mu \right)^{\frac{1}{p}}$$
for each open ball $B \subeq X$ centered at a point of $K$ of radius less than $R$.
\end{theorem}

In the above result, it is assumed \emph{a priori} that the mapping $f$ is continuous. In fact, one could instead assume only that the mapping $f$ is locally integrable in a suitable sense; it then follows that $f$ has a continuous representative satisfying the desired conclusion.

It may occur that the quantity $\sigma$ in Theorem \ref{PoinGivesSob} is necessarily strictly larger than one \cite[Section 9]{SobMet}. This is an inconvenience when working with coverings. In many situations the following statement, which we learned from Koskela and Z\"urcher, ameliorates this problem.

\begin{lemma}\label{shrink}
Assume that $(X,d,\mu)$ is a locally doubling metric measure space and let $1\leq q<p$ and $0<\tau\leq 1$. For each $g \in \Lloc^{p}(X)$, there is a Borel function $\til{g} \in \Lloc^{p/q}(X)\subeq \Lloc^1(X)$ such that for each compact set $K\subseteq X$ there exists a constant $C\ge 1$ and a radius $R>0$ so that
$$\mint_{B(x,r)}g^{q}\ d\mu \leq C \mint_{B(x,\tau r)} \til{g} \ d\mu$$
for each $x \in K$ and each $0<r<R$.
\end{lemma}

\begin{proof}
Let $y \in B(x,\tau r)$.  Then
$$\int_{B(x,r)}g^{q} \ d\mu \leq \mu(B(y,(1+\tau)r))M(g^{q})(y),$$
where $M(g^{q}) \in \Lloc^{p/q}(X)$ is a suitably restricted maximal function of $g^{q}$ \cite[Chapter 2]{LAMS}.  Integrating the above inequality over $B(x,\tau r)$ yields
\begin{align*}\mu(B(x,\tau r)) \int_{B(x,r)}g^{q} \ d\mu
& \leq \int_{B(x,\tau r)} \mu(B(y,(1+\tau)r))M(g^{q})(y) \ d\mu(y) \\
& \leq \mu(B(x, (1+2\tau)r)) \int_{B(x,\tau r)} M(g^{q})(y). \end{align*}
The local doubling condition now implies that $\til{g}=M(g^{q})$ satisfies the requirements of the statement.
\end{proof}

H\"older's inequality, Theorem \ref{PoinGivesSob},  and Lemma \ref{shrink} imply the following statement, which will be the form of Morrey's estimate most frequently applied in this paper.

\begin{proposition}\label{better Morrey}
Assume that $X$ is locally $Q$-homogeneous, $Q \geq 1$, and supports a local $Q$-Poincar\'e inequality. Let $Q<q<p$. If $f \colon X \to Y$ is a continuous mapping with an upper gradient in $\Lloc^p(X)$, then there exists a Borel function $g \in \Lloc^{p/q}(X)\subeq \Lloc^1(X)$ such that for each compact set $K \subeq X$, there exists a constant $C \geq 1$ and a radius $R>0$ such that
$$\diam f(B) \leq C \diam B \left(\mint_{B(x,r/5)} g\ d\mu\right)^{1/q},$$
for each $x \in K$ and $0<r<R$.
\end{proposition}

\section{Universal bounds on dimension distortion}\label{Universal section}

In this section, we prove Theorems \ref{Kaufman metric intro} and \ref{universal-sharp}.

%The proof of Theorem \ref{Kaufman-intro}, which provides universal bounds on dimension distortion under Sobolev mappings on Euclidean domains, can be summarized as follows. We consider an arbitrary subset $E$ of the domain with finite $\Hdim^s$ measure. We then find a covering of $E$ by arbitrarily small essentially disjoint dyadic cubes so that the sum of the side lengths raised to the dimension of $E$ is finite. In Euclidean space, the Morrey estimate  applies to cubes, and so we may estimate the diameter of the image of each cube. This gives rise to an estimate on the $\frac{ps}{p-Q+s}$-dimensional measure of $f(E)$ via H\"older's inequality and the essential disjointness of the dyadic cubes.
%
%The main difficulty in adapting this argument to the metric setting is the absence of essentially disjoint coverings by sets on which the Morrey estimate applies. It is not clear if the Morrey estimate applies to the Christ cubes, which are the usual substitute for a dyadic cube structure in a metric measure space \cite{Christ}.  We overcome this difficulty by employing Proposition \ref{better Morrey} and a covering theorem. The price to pay is that we only obtain the dimension estimate \eqref{Kaufman-estimate-intro} rather than the vanishing of the appropriate Hausdorff measure as in Theorem \ref{Kaufman-intro}. Theorem~ \ref{Q-regular} remedies this when $E$ is assumed to be locally Ahlfors regular.

\begin{proof}[Proof of Theorem \ref{Kaufman metric intro}]
We first consider the case $\dim E < Q$.  Fix $t \in (\dim E, Q)$, and choose $q \in (Q,p)$ so close to $p$ that
$$\frac{p \dim E}{p-Q+\dim E}< \frac{q t}{q-Q+t} < \frac{pt}{p-Q+t}.$$
Let $\alpha = \tfrac{qt}{q-Q+t} \in (0,Q)$.

The countable subadditivity of Hausdorff measure allows us to assume that $E$ is contained in a ball $B_0$, which has compact closure.  Hence, by Proposition \ref{better Morrey}, there is a constant $C \geq 1$, a radius $R>0$, and a Borel function $g \in \Lloc^{\frac{p}{q}}(X)$ such that for every $x \in E$ and $r<R$,
\begin{equation}\label{Morrey 2}\diam f(B(x,r)) \leq Cr \left(\mint_{B(x,r/5)} g\right)^{1/q}.
\end{equation}

Let $\ep, \ep'>0$. Since $t>\dim E$, it holds that $\Hdim^t(E)=0$.  Hence, it follows from the definitions and the $5B$-covering theorem \cite[Theorem 1.2]{LAMS} that there is a collection $\{B(x_k, r_k)\}_{k \in \nats}$ of balls centered in $E$ such that
\begin{itemize}
\item $\sum_{k \in \nats} r_k^t < \ep$,
\item $ \sup_{k \in \nats} r_k< \ep'$,
\item $E \subeq \bigcup_{k \in \nats} B(x_k,r_k) \subeq B_0,$
\item $B(x_k,r_k/5) \cap B(x_j,r_j/5) = \emptyset$ if $j\neq k$.
\end{itemize}
Since $f$ is uniformly continuous on small sets, choosing $\ep'$ small enough ensures that for all $k \in \nats,$
$$\diam f(B(x_k,r_k)) < \ep.$$
Reducing $\ep'$ to be less than $R$ if necessary, and using \eqref{Morrey 2} and local homogeneity, we see that
\begin{align*}
\Hdim^{\alpha}_{\ep}(f(E))  & \leq \sum_{k \in \nats} (\diam f(B(x_k,r_k)))^\alpha \\
& \lesssim \sum_{k \in \nats}r_k^{\left(1-\frac{Q}{q}\right)\alpha}\left(\int_{B(x_k,r_k/5)} g\ d\mu\right)^{\frac{\alpha}{q}}.
\end{align*}
Applications of H\"older's inequality and the disjointness assumption now yield
\begin{align*}\Hdim^{\alpha}_{\ep}(f(E)) &\lesssim \left( \sum_{k \in \nats}r_k^{t}\right)^{1-\frac{\alpha}{q}}\left(\int_{B_0}g\ d\mu\right) ^{\frac{\alpha}{q}} \\
& \lesssim \ep^{1-\frac{\alpha}{q}}\left(\int_{B_0}g\ d\mu\right) ^{\frac{\alpha}{q}}.\end{align*}
Since $g \in \Lloc^{p/q}(X) \subset \Lloc^1(X)$, letting $\ep$ tend to zero shows that $\dim f(E)\leq \alpha$. Letting $t$ tend to $\dim E$ now yields the desired result.

We now consider the case $\dim E = Q$. Choose any $q \in (Q,p)$.  As before, we use Proposition \ref{better Morrey} to find a constant $C \geq 1$, a radius $R>0$, and a Borel function $g \in \Lloc^{\frac{p}{q}}(X)$ such that
$$\diam f(B(x,r)) \leq C r \left(\mint_{B(x,r/5)} g \ d\mu\right)^{1/q}$$
for every $x \in E$ and $r<R$. This implies that
$$\diam f(B(x,r)) \leq C r (M(g)(x))^{1/q}.$$
Since $M(g) \in \Lloc^{\frac{p}{q}}(X)$, there is a sequence $E_1 \supeq E_2 \supeq \hdots$ of subsets of $E$ such that for each $n \in \nats$, there is a number $L_n \geq 1$ such that
$$\diam f(B(x,r)) \leq L_n r$$
for each $x \in E \bslash E_n$ and $r<R$, and $\mu(E_n) \leq 1/n$.  Then $N = \bigcap_n E_n$ satisfies $\mu(N)=0$, and
$$\dim f\left(E\bslash N \right) \leq \dim E=Q.$$
Note that local $Q$-homogeneity implies that $\Hdim^Q\lfloor E$ is absolutely continuous with respect to $\mu$. Hence $\Hdim^Q(N)=0$. It suffices to show that $\Hdim^Q(f(N))=0$ as well; the proof of this is analogous to the proof in the previous case and is left to the reader. This also proves the final statement of the Theorem. 
\end{proof}

In order to reach the stronger conclusion that the image of a given set has zero measure in the appropriate dimension, we assume the set has additional structure:

\begin{theorem}\label{Q-regular}Assume the notation and hypotheses of Theorem \ref{Kaufman metric intro}, and further assume that $\mu(E) = 0$ and there is $0 \leq t < Q$ such that $\Hdim^t\lfloor{E}$ is Ahlfors $t$-regular. Then $$\Hdim^{\frac{pt}{p-Q+t}}(f(E))=0.$$
\end{theorem}

A simple modification of the proof of \cite[Theorem~1]{Kaufman}, which is the Euclidean version of Theorem~\ref{Kaufman metric intro}, shows that Theorem~\ref{Q-regular} is true when the assumption of Ahlfors regularity is replaced by the assumption that $E$ has the following covering property for sufficiently large values of the parameter $\sigma$:

\begin{definition}\label{evenly coverable} Let $\sigma \geq 1$. A subset $E$ of a metric space $(X,d)$ is \emph{$\sigma$-evenly coverable} if there exists a constant $C \geq 1$ such that for all sufficiently small $\ep>0$, there exists a cover $\{B(x_k,r_k):k \in \nats\}$ of $E$ by balls centered in $E$ such that
\begin{itemize}
\item[i)]$\sup_{k \in \nats} r_k \leq \ep$,
\item[ii)] $\sum_{k \in \nats} r_k^{\dim E} \leq C$,
\item[iii)] $\sup_{x \in X} \sum_{k \in \nats} \chi_{B(x_k,\sigma r_k)}(x)  \leq  C$.
\end{itemize}
\end{definition}

Hence, Theorem \ref{Q-regular} follows from the following proposition.

\begin{proposition}\label{Q reg cover}  Let $(X,d)$ be a metric space, and let $E \subeq X$ be a bounded Ahlfors $t$-regular subset of $X$. Then $E$ is $\sigma$-evenly coverable for every $\sigma \geq 1$.
\end{proposition}

\begin{proof} Let $\sigma \geq 1$. We assume that there is a constant $K \geq 1$ such that for any $r \leq 2 \diam E$ and $x \in E$, 
$$\frac{r^t}{K} \leq \Hdim^t_{X}(B(x,r) \cap E) \leq Kr^t.$$
In particular, this implies that $\Hdim^t_X(E)<\infty$.

Let $\ep>0$, and consider a maximal $\ep$-separated set $\{x_1,\hdots,x_N\}$ in $E$. Then $\{B(x_k,\ep)\}_{k=1}^N$ covers $E$, while $\{B(x_k,\frac{\ep}{2})\cap E\}_{k=1}^N$ is disjoint. Thus 
\begin{equation}\label{calculate} \sum_{k=1}^N \ep^t \leq \sum_{k=1}^N 2^tK\Hdim^t_{X}\left(B\left(x_k,\frac{\ep}{2}\right)\cap E \right) \leq 2^tK\Hdim^t_X(E).\end{equation}
This shows that $\{B(x_k,\ep)\}_{k=1}^N$ satisfies conditions (i) and (ii) in the definition of $\sigma$-even coverability. 

Suppose that 
$$\bigcap_{k \in I} B(x_k,\sigma\ep) \neq \emptyset$$
where $I$ is a subset of $\{1,\hdots,N\}$. To see that condition (iii) is verified, we must show that the cardinality of $I$ is bounded above by a number that does not depend on $\ep$. Let $i_0 \in I$. Then our assumption yields
$$\bigcup_{k \in I} B(x_k,\sigma \ep)\subeq B(x_{i_0},3\sigma\ep).$$
This implies that 
\begin{align*} (\card{I})\ep^t  &\leq \sum_{k \in I} 2^tK\Hdim^t_{X}\left(B\left(x_k,\frac{\ep}{2}\right)\cap E \right) \\ &\leq 2^tK\Hdim^t_X(B(x_{i_0},3\sigma\ep)\cap E) \\ &\leq K^2(6\sigma \ep)^t.
\end{align*}
The desired bound on the cardinality of $I$ follows. \end{proof}

We now turn to the question of sharpness in Theorem \ref{Kaufman metric intro}. We first prove Theorem \ref{universal-sharp} on the existence of mappings with $L^p$ upper gradients exhibiting optimal dimension increase. Note that the $Q$-Poincar\'e inequality is {\it not} assumed in Theorem \ref{universal-sharp}.

\begin{proof}[Proof of Theorem \ref{universal-sharp}]
The proof is a modified version of those appearing in \cite{Kaufman} and \cite{BMT}. The main novelty is that we employ maximal separated sets in place of dyadic cubes; this in fact simplifies the proof. Let $(X,d,\mu)$ be a locally $Q$-Ahlfors regular metric measure space and let $E\subseteq X$ be a compact set with $\Hdim^s(E)>0$ for some $0\le s\le Q$. We may assume without loss of generality that $\diam E < 1$. By Frostman's Lemma \cite[Theorem 8.17]{Mattila}, there is a finite and nontrivial Borel measure $\nu$ supported on $E$ such that
\begin{equation}\label{Frostman} \nu(B(x,r)) \leq r^s\end{equation}
for each $x \in X$ and $r>0$.

For each $n \in \nats$, let $X_n$ be a maximal $2^{-n}$-separated set in $E$; we may assume that $X_1 \subeq X_2 \subeq \hdots$.  Define
$$\mathcal{Q}_n = \{B(z,2^{-n}): z \in X_n\} \ \text{and}\ \mathcal{Q}= \bigcup_{n \in \nats} \mathcal{Q}_n.$$
Note that each $\mathcal{Q}_n$ is finite. The local doubling condition on $X$ implies that there is a constant $C \geq 1$ such that
\begin{equation}\label{bdd overlap}\sum_{B \in \mathcal{Q}_n} \chi_{100B}(x) \leq C\end{equation}
for all $x \in X$ and all $n \in \nats$.

For each $B \in \mathcal{Q}$, we may find a Lipschitz function $\psi_B \colon X \to [0,1]$ such that $\psi_B|_{\ovl{B}} =1$, the support of $\psi_B$ is contained in $2B$, and
$$\Lip \psi_B \lesssim (\diam B)\inv.$$
Here $\Lip f$ denotes the pointwise Lipschitz constant of the function $f$, defined by
$$
\Lip f(x) = \limsup_{y\to x} \frac{|f(y)-f(x)|}{d(x,y)}.
$$

Let $\xi\colon \mathcal{Q} \to \ovl{B}_{\reals^N}(0,1)$ be a function. For each $n \in \nats$, define $f_{\xi,n} \colon X \to \reals^N$ by
$$f_{\xi,n}(x) = \sum_{B \in \mathcal{Q}_n} \nu(100B)^{1/\alpha}\psi_B(x)\xi(B).$$
Now, define $f_\xi \colon X \to \reals^N$ by
$$f_\xi(x) = \sum_{n \in \nats}(1+n)^{-2}f_{\xi,n}(x).$$

Then $f_\xi$ is continuous and bounded. Since $f_{\xi,n}$ is locally Lipschitz, the function $\Lip f_{\xi,n}$ is an upper gradient of $f_{\xi,n}$ \cite{Cheeger}. We claim that the sequence of norms $\{||\Lip f_{\xi,n}||_{{\rm L}^p}\}_{n \in \nats}$ is bounded. Using the bounded overlap condition \eqref{bdd overlap}, the Frostman condition \eqref{Frostman}, and Ahlfors $Q$-regularity, we calculate that
\begin{align*}||\Lip f_{\xi,n}||_{{\rm L}^p}^p & \lesssim 2^{np} \sum_{B \in \mathcal{Q}_n}  \nu(100B)^{p/\alpha}\mu(2B)\\
& \lesssim 2^{n\left((p-Q)-s\left(\frac{p}{\alpha} -1\right)\right)}\sum_{B \in \mathcal{Q}_n} \nu(100B).\end{align*}
Our choice of $\alpha$ implies that
$$(p-Q)-s\left(\frac{p}{\alpha} -1\right) = 0,$$
and hence another application of the bounded overlap condition \eqref{bdd overlap} shows that
$$||\Lip f_{\xi,n}||_{{\rm L}^p}^p \lesssim \sum_{B \in \mathcal{Q}_n} \nu(100B) \lesssim \nu(X).$$
These facts imply that $\Lip f_{\xi}\in{\rm L}^p$ is an upper gradient of $f_\xi$. 

We now choose the vectors $\xi_B$ randomly. More precisely, we assume that the functions $\{\xi_B\}_{B \in \mathcal{Q}}$ are independent random variables distributed according to the uniform probability distribution on the closed unit ball $\ovl{B}_{\reals^N}(0,1)$, and hence the resulting function $\xi \colon \mathcal{Q} \to \ovl{B}_{\reals^N}(0,1)$ can also be considered as random variable; the expected value of this random variable is denoted by $\E_\xi$.  We claim that $\dim f_\xi(E) \geq \alpha$ almost surely. The desired result follows from this claim.

For $t>0$ denote by $I_t(\lambda)$ the {\it $t$-energy} of a compactly supported Radon measure $\lambda$ on a metric space $X$, i.e.
$$I_t(\lambda) = \iint d(x,y)^{-t} \, d\lambda(x) \, d\lambda(y).$$
If $I_t(\lambda)$ is finite, then the Hausdorff dimension of the support of $\lambda$ is at least $t$ \cite[Theorem 8.7]{Mattila}.

 We will prove that for every $\alpha' < \alpha$,
$$\E_\xi \left( I_{\alpha'} ( (f_\xi)_\#(\nu \lfloor E) ) \right) < \infty,
$$
which implies that $\dim f_\xi(E) \ge \alpha'$ almost surely; letting $\alpha'$ tend to $\alpha$ will complete the proof.

By the Fubini--Tonelli theorem, it suffices to prove that
\begin{equation} \label{energy}
\iint \E_\xi \left( |f_\xi(x) - f_\xi(y)|^{-\alpha'} \right) \, d\nu(y) \, d\nu(x) < \infty.
\end{equation}
We write
$$
f_\xi(x) - f_\xi(y) = \sum_{B \in \mathcal{Q}} c_B(x,y) \, \xi_B,
$$
where
$$
c_B(x,y) = (1+n)^{-2} \, \nu(100B)^{\frac{1}{\alpha}} ( \psi_B(x) - \psi_B(y) ) \qquad \mbox{when }B \in \mathcal{Q}_n.
$$
We denote by $||c(x,y)||_\infty$ the maximum of the set of numbers $\{c_B(x,y)\}_{B \in \mathcal{Q}}.$
Note that $||c(x,y)||_\infty = |c_{Q_0}(x,y)|$ for some $Q_0 \in \mathcal{Q}$, since
$\sum_{B \in \mathcal{Q}}|c_B(x,y)|$ is finite. By \cite[Lemma 4.4]{BMT}, it holds that
$$\E_\xi \left( |f_\xi(x) - f_\xi(y)|^{-\alpha'} \right) \lesssim ||c(x,y)||_\infty^{-\alpha'}.$$
In view of this, it remains to show that
$$
\int_E \int_E ||c(x,y)||_\infty^{-\alpha'} \, d\nu(y) \, d\nu(x) < \infty.
$$
We will in fact show the stronger statement
$$
\sup_{x \in E} \int ||c(x,y)||_\infty^{-\alpha'} \, d\nu(y) <\infty.
$$
Since $\nu(E)<\infty$ this suffices.

Fix $x \in E$. For each $y \in E$, define $n(y) \in \nats$ by
$$2^{-n(y)+2} \leq d(x,y) < 2^{-n(y)+3}.$$
Choose a ball $B \in \mathcal{Q}_{n(y)}$ that contains $x$. Then $y \in 100B \bslash 2B$, and so
$$
||c(x,y)||_\infty \ge |c_{B}(x,y)| = (1+n(y))^{-2} \nu(100B)^{\frac{1}{\alpha}}.
$$
For each $n \in \nats$, denote by $E_n$ the set of points $y\in E$ for which $n(y)=n$. As above $E_n \subset 100B_n\bslash 2B_n$, where $B_n \in \mathcal{Q}_{n}$ contains $x$. Thus, by the above argument and the Frostman estimate \eqref{Frostman},
\begin{align*}
\int ||c(x,y)||_\infty^{-\alpha'} \, d\nu(y) &= \sum_{n \in \nats} \int_{E_n} ||c(x,y)||_\infty^{-\alpha'}\ d\nu(y) \\
& \le \sum_{n \in \nats} n^{2\alpha'} \nu(100B_n)^{1-\frac{\alpha'}{\alpha}}\\
& \lesssim \sum_{n \in \nats} n^{2\alpha'}  2^{-ns\left(1-\tfrac{\alpha'}{\alpha}\right)}.\end{align*}
Since $\alpha'<\alpha$, the final sum converges.

We now  show that the set of Sobolev mappings that distort the Hausdorff dimension of a given set in the maximal way is prevalent, in the sense of Hunt--Sauer--Yorke \cite{hunt-sauer-yorke1992}, \cite{sauer-yorke1997} (see also \cite{ott-yorke2005}, \cite{Christensen}, and \cite{Aronszajn}). 

To recall the notion of prevalence, let $\mathbb{B}$ be a complete metric vector space (typically infinite dimensional). A compactly supported Borel measure $\lambda$ on $\mathbb{B}$ is said to be {\it transverse} to a Borel set $S\subseteq \mathbb{B}$ if $\lambda(S+x)=0$ {\it for every } $x\in \mathbb{B}$. A set $S'\subseteq  \mathbb{B}$ is called to be {\it shy} if there exists a Borel set $S$ such that $S'\subseteq S\subseteq \mathbb{B}$ and a Borel measure $\lambda$ that is transverse to $S$. Using convolutions of measures it can be checked that the countable union of shy sets is again shy \cite{hunt-sauer-yorke1992, ott-yorke2005}. A set $Y\subseteq \mathbb{B}$ is called {\it prevalent} if its complement $S=\mathbb{B}\setminus Y$ is shy. Clearly, the countable intersections of prevalent sets is again prevalent and prevalent sets are dense in $\mathbb{B}$. If $\mathbb{B}=\R^{n}$, then $Y\subseteq \R^{n}$ is prevalent if and only if it is a full Lebesgue measure set. The concept of prevalence has been introduced as a measure-theoretic notion of genericity in infinite dimensional spaces, especially function spaces. We will use this notion for the Newtonian--Sobolev space $\mathbb{B}={\rm N}^{1,p}(X;\R^{N})$.

We consider a compact subset $E \subeq X$ such that $\Hdim^s(E)>0$, and wish to show that the set of Newtonian Sobolev mappings $f \in {\rm N}^{1,p}(X; \reals^N)$ with the property that $\dim f(E) \geq \alpha$ is prevalent.  Notice first that it is enough to show that
\begin{equation} \label{alpha-prev}
\dim f(E) \geq \alpha' \ \text{for a prevalent subset} \   W_{\alpha'} \  \text{of maps in} \ {\rm N}^{1,p}(X; \R^{N})
\end{equation}
for each $\alpha' < \alpha$. Indeed, assuming that this is true we obtain prevalent subsets $W_{n}$, $n \in \nats$, for which
\begin{equation} \label{n-prevalent}
\dim f(E) \geq \alpha - \frac1n \quad \mbox{for every} \   f\in W_{n}.
\end{equation}
Now set $W=\bigcap_{n}W_{n}$, which is again prevalent in ${\rm N}^{1,p}(X;\R^{N})$ as the countable intersection of prevalent sets. Letting $n \to \infty$ in \eqref{n-prevalent} we obtain that $\dim f(E) \geq \alpha$ for all $f \in W$.

Note that by \eqref{energy} there exists a continuous mapping $g \in {\rm N}^{1,p}(X;\R^{N})$ with the property that
\begin{equation} \label{g-energy}
\iint  |g(x) - g(y)|^{-\alpha'} \, d\nu(y) \, d\nu(x) < \infty.
\end{equation}
Here $\nu$ denotes the Frostman measure on $E$, as in \eqref{Frostman}.

Statement \eqref{alpha-prev} is implied by the following lemma.
\begin{lemma}\label{prev-lemma}
Let $g\in {\rm N}^{1,p}(X;\R^{N})$ satisfy \eqref{g-energy}. Denote by $Z$ the set of all $N\times N$ matrices with entries less than or equal to one in absolute value. Then for all $f_{0}\in {\rm N}^{1,p}(X;\R^{N})$ the function $f_{L}= f_{0}+Lg$ satisfies
\begin{equation} \label{L-dimension}
\dim f_{L}(E) \geq \alpha' \quad \mbox{for almost every} \ L \in Z.
\end{equation}
\end{lemma}

The proof of Lemma \ref{prev-lemma} is similar to the proof of Proposition 3.2 from \cite{hunt-kaloshin1997}. For the convenience of the reader we provide a sketch. The idea is again to use energy estimates: we shall show that
\begin{equation} \label{fL-energy}
\iint  |f_{L}(x) - f_{L}(y)|^{-\alpha'}  \, d\nu(y) \, d\nu(x) < \infty,  \quad \mbox{for almost every} \ L \in Z,
\end{equation}
which in turn will follow from the boundedness of the triple integral
\begin{equation} \label{triple-energy}
\int_{Z}\iint  |f_{L}(x) - f_{L}(y)|^{-\alpha'}  \, d\nu(y) \, d\nu(x) \, dL < \infty.
\end{equation}

To prove \eqref{triple-energy} we will use the following

\begin{lemma} \label{int-est}
Let $\Phi$ be a linear transformation from the set of $N\times N$ matrices to $\R^{N}$ and let $b\in \R^{N}$ be a fixed vector. Assume that the image of
$Z$ under $\Phi$ contains a cube of width $\delta$ in $\R^{N}$. Then for $\alpha' < N$ we have
\begin{equation} \label{Z-int}
\int_{Z}\frac{dL}{|\Phi (L) +b|^{\alpha'}} \leq \frac{C}{\delta^{\alpha'}},
\end{equation}
where $C$ is a constant depending only on $N$ and $\alpha'$.
\end{lemma}

A proof of Lemma \ref{int-est} may be found in \cite[Lemma 3.3]{hunt-kaloshin1997} and \cite[Lemma 2.6]{sauer-yorke1997}.

We apply Lemma \ref{int-est} to $b:= f_{0}(x)-f_{0}(y)$ and $\Phi(L):=L(g(x)-g(y))$, noting that $\Phi(Z)$ contains a cube of width comparable to $\delta = |g(x)-g(y)|$. This implies
\begin{equation} \label{appl-claim}
\int_{Z}  |f_{L}(x) - f_{L}(y)|^{-\alpha'}  dL\leq \frac{K}{|g(x)-g(y)|^{\alpha'}} .
\end{equation}
By the Fubini--Tonelli theorem we can estimate the integral in \eqref{triple-energy} using \eqref{appl-claim} and \eqref{g-energy} as follows:
\begin{align*} \int_{Z}\iint   |f_{L}(x) - f_{L}(y)|^{-\alpha'}  \, d\nu(y) \, d\nu(x) \, dL  \\
= \iint  \int_{Z} |f_{L}(x) - f_{L}(y)|^{-\alpha'}  \, dL \, d\nu(y) \, d\nu(x)
 \\ \leq K \iint  |g(x) - g(y)|^{-\alpha'}  \, d\nu(y) \, d\nu(x) < \infty.
\end{align*}
This finishes the proof of Lemma \ref{prev-lemma} and completes the proof of Theorem~\ref{universal-sharp}.
\end{proof}

\section{Regular foliations of a metric space}\label{reg foliations section}

In this section we discuss bounds on dimension increase under Sobolev mappings for leaves in an $s$-foliation of a metric space. In particular, we prove Theorem \ref{foliation}. Following the proof, we provide some comments regarding the limitations of that theorem and alternate methods to derive similar estimates.

The first step in the proof of Theorem \ref{foliation} is the following lemma, which enables us to use Frostman's lemma.

\begin{lemma}\label{union of compacts}
Assume the notation and hypotheses of Theorem \ref{foliation}. Then, for any compact set $K \subeq X$, the set
$$E_\alpha = \{a \in W: \Hdim^\alpha(f(\pi\inv(a) \cap K))>0\}$$
is a countable union of compact sets.
\end{lemma}

\begin{proof}
As $W$ is assumed to be proper, it suffices to show that $E_\alpha$ is a countable union of closed sets. Since the $\alpha$-dimensional Hausdorff measure and the $\alpha$-dimensional Hausdorff content have the same null sets, it suffices to show that for each $n \in \nats$, the set
$$E_\alpha(n)= \left\{a \in W: \Hdim^{\alpha}_{\infty}(f(\pi\inv(a) \cap K))\geq \frac{1}{n}\right\}$$
is closed. Let $\{a_j\}_{j\in \nats} \subeq E_\alpha(n)$ be a sequence converging to a point $a \in W$.  Since $f$ and $\pi$ are continuous, for every $\ep>0$, there is an index $j(\ep) \in \nats$ such that if $j \geq j(\ep)$, then
$$f(\pi\inv(a_j) \cap K) \subeq \nbhd_{Y}(f(\pi\inv(a) \cap K), \ep).$$
If $a \notin E_\alpha(n)$, then there is a cover $\{B_Y(y_i,r_i)\}_{i \in \nats}$ of $f(\pi\inv(a) \cap K)$ by open balls such that
$$\sum_{i \in \nats} r_i^\alpha < \frac{1}{n}.$$
Since $f(\pi\inv(a) \cap K)$ is compact, we may find $\ep>0$ such that the neighborhood $\nbhd_Y(f(\pi\inv(a) \cap K),\ep)$ is also covered by $\{B_Y(y_i,r_i)\}_{i \in \nats}$. This implies that
$$\Hdim^{\alpha}_{\infty}(f(\pi\inv(a_j) \cap K)) < \frac{1}{n}$$
for all $j \geq j(\ep)$, which yields the desired contradiction.
\end{proof}

\begin{proof}[Proof of Theorem \ref{foliation}]
For ease of notation, denote
$$\beta = (Q-s)-p\left(1-\frac{s}{\alpha}\right).$$
As we only consider the case that $\alpha > s$, it suffices to show that 
$$\dim\{a \in W: \Hdim^{\alpha}(f(\pi\inv(a)))>0\} \leq \beta.$$
Let $K$ be an arbitrary compact subset of $X$, and set
$$E_\alpha = \{a \in W: \Hdim^\alpha(f(\pi\inv(a) \cap K))>0\}.$$
By the stability of Hausdorff dimension under countable unions and the properness assumption on $X$, the desired result will follow if $\dim E_{\alpha} \leq \beta$.

Define a function $\phi \colon [\beta,\infty) \times (Q,p] \to \reals$ by
$$\phi(t,q) =  \left(\left(1-\frac{Q}{q}\right)\alpha + t-s\left(1-\frac{\alpha}{q}\right)\right)\frac{q}{q-\alpha}.$$
Our assumptions imply that $\alpha < Q$, and so $\phi$ is continuous.

Now suppose, by way of contradiction, that $\dim E_\alpha > \beta$, and let $t \in (\beta, \dim E_{\alpha})$. Note that $\phi(\beta,p) = \beta$.  Moreover, since $t>\beta$,
\begin{equation}\label{phi ineq} \phi(t,p) = \beta + \left(t-\beta\right) \frac{p}{p-\alpha} > t.\end{equation}
Since $\phi$ is continuous, we may find $q \in (Q,p) $ such that
\begin{equation}\label{q def} t' :=\phi(t,q) > t.\end{equation}

The countable stability of Hausdorff dimension and Lemma \ref{union of compacts} allows us to reduce to the case that $E_\alpha$ is compact. Since $t<\dim E_\alpha$, it holds that $\Hdim^t(E_\alpha)=\infty$, and so by \cite[Theorem~8.19]{Mattila} there exists a compact subset $E \subeq E_\alpha$ such that $0<\Hdim^t(E)<\infty$. Frostman's lemma \cite[Theorem~8.17]{Mattila} yields a nonzero Borel measure $m$ supported on $E$ with the property that  the upper mass bound $m(B_W(a,r))\leq r^t$ is valid for every $a \in W$ and all $r>0$.

Let $\del, \ep > 0$. As $t'>t$, it holds that $\Hdim^{t'}(E) = 0$, and so we may find a countable cover $\{B_W(a_i,r_i)\}_{i \in \nats}$ of $E$ such that $r_i < \del$ for all $i \in \nats$ and
\begin{equation}\label{small r}
\sum_{i \in \nats} r_i^{t'} < \ep.
\end{equation}
By choosing $\del$ sufficiently small, we may apply the foliation condition to each ball $B_W(a_i,r_i)$, producing a constant $C \geq 1$ and a cover $\{B_{i,j}:j=1,\ldots,N_i\}$ of $\pi\inv(B_W(a_i,r_i)) \cap K$ where
\begin{equation}\label{N-i-estimate}
N_i \leq Cr_i^{-s}
\end{equation}
and the radius of the ball $B_{i,j}$ is $Cr_i$. We apply the $5B$ covering theorem to the doubly indexed collection $\{B_{i,j} \colon i \in \nats , j=1,\hdots, N_i \}$ to produce a set $I \subeq \nats$ and for each $i \in I$ a (possibly empty) set $J_i \subeq \{1,\hdots, N_i\}$ so that
\begin{equation}\label{inverse cover}
\pi\inv(E) \cap K \subeq  \pi\inv\left(\bigcup_{i \in \nats} B_W(a_i,r_i)\right) \cap K \subeq \bigcup_{i \in I} \bigcup_{j \in J_i} 5B_{i,j},
\end{equation}
and so that $B_{i,j} \cap B_{i',j'} = \emptyset$ whenever $(i,j)\neq (i',j').$

Let $\tau > 0$. By the uniform continuity of $f$ on compact sets, if $\del$ is sufficiently small, then for any $a \in E$,
\begin{equation}\label{char function estimate}
\Hdim^{\alpha}_{\tau}(f(\pi\inv(a) \cap K)) \leq \sum_{i \in I}\sum_{j \in J_i} \chi_{\pi(B_{i,j})}(a)(\diam f(B_{i,j}))^\alpha.\end{equation}
Further reducing $\del$ if necessary so that we may apply Morrey's estimate (in the form of Proposition \ref{better Morrey}) and the local $Q$-homogeneity condition, we find a Borel function $g \in \Lloc^{\frac{p}{q}}$ such that
\begin{equation}\label{morrey-estimate}
\diam f(B_{i,j})\lesssim r_i^{1-\frac{Q}{q}} \left(\int_{(1/5)B_{i,j}} g \ d\mu\right)^{\frac{1}{q}}
\end{equation}
for each $i \in I$ and $j \in J_i$. Integrating \eqref{char function estimate} and using \eqref{morrey-estimate}, we see that
\begin{align*}
\int_{E} \Hdim^{\alpha}_{\tau}(f(\pi\inv(a) \cap K)) \ dm(a) & \lesssim \sum_{i \in I}\sum_{j \in J_i} m(\pi(B_{i,j})) (\diam f(B_{i,j}))^\alpha \\
& \lesssim \sum_{i \in I}\sum_{j \in J_i} m(\pi(B_{i,j}))r_{i}^{\left(1-\frac{Q}{q}\right)\alpha} \left(\int_{(1/5)B_{i,j}} g\right)^\frac{\alpha}{q}.\end{align*}
Here one may consider the integral as an upper integral to avoid tedious measurability issues. 

Since $\pi$ is Lipschitz on the compact set $$K'=\ovl{\bigcup_{i \in I, j \in J_i}B_{ij}},$$
the Frostman condition on $m$ implies that
$$m(\pi(B_{i,j})) \lesssim r_i^t,$$
again provided that $\del$ is small enough.  This estimate, together with \eqref{N-i-estimate} and two applications of H\"older's inequality, implies that
\begin{align*}
\int_{E} \Hdim^{\alpha}_{\tau}(f(\pi\inv(a) \cap K)) & \ dm(a)  \lesssim \sum_{i \in I}\sum_{j \in J_i} r_i^{\left(1-\frac{Q}{q}\right)\alpha + t} \left(\int_{(1/5)B_{i,j}} g  \ d\mu\right)^{\frac{\alpha}{q}}  \\
& \lesssim \sum_{i \in I} r_i^{\left(1-\frac{Q}{q}\right)\alpha + t-s\left(1-\frac{\alpha}{q}\right)}\left(\int_{\bigcup_{j \in J_i}(1/5)B_{i,j}} g \  d\mu\right)^{\frac{\alpha}{q}} \\
&\lesssim \left(\sum_{i \in I} r_i^{t'}\right)^{1-\frac{\alpha}{q}} \left(\int_{K'} g \  d\mu\right)^{\frac{\alpha}{q}}.
\end{align*}
In light of \eqref{small r} and the local integrability of $g$, we conclude that
$$\int_{E} \Hdim^{\alpha}_{\tau}(f(\pi\inv(a) \cap K))  \ dm(a) \lesssim \ep.$$
Letting $\ep\to 0$ implies that for $m$-almost every $a \in E$,
$$\Hdim^{\alpha}_{\tau}(f(\pi\inv(a) \cap K))=0=\Hdim^{\alpha}(f(\pi\inv(a) \cap K)).$$
This is a contradiction, as $m$ is supported on $E$ and $E$ is a subset of $E_\alpha$.
\end{proof}

\begin{remark}\label{limitations}
The wide generality allowed by the definition of a metric foliation comes at a price; the estimate of Theorem \ref{foliation} is not always optimal. Assume the hypotheses of Theorem \ref{foliation}. Denote
$$\hat{s} = \sup_{a \in W} \dim \pi\inv(a).$$
As noted in the introduction, it could be that $\hat{s} < s$; see subsection \ref{Heis foliation} below for an example.
By the universal dimension distortion bounds given in Theorem \ref{Kaufman metric intro},
$$
\{a \in W: \Hdim^\alpha(f(\pi\inv(a)))>0\} = \emptyset
$$
whenever $\alpha \geq \tfrac{p\hat{s}}{p-Q+\hat{s}}$. When $\alpha = \tfrac{p\hat{s}}{p-Q+\hat{s}}$ and $\hat{s}<s$, it holds that 
$$
(Q-s) - p\left(1-\tfrac{s}{\alpha}\right) = (p-Q)\left(\frac{s}{\hat{s}} - 1 \right) > 0,
$$
and so there is room for a possible improvement to the conclusion of Theorem \ref{foliation} in this situation.

The correct estimate in the case when $\alpha$ lies in the range $[\hat{s},s)$ is also unclear. Note that the proof of Theorem \ref{foliation} does not apply when $\alpha<s$. If $\dim W \leq Q-s$, as is the case for all of the examples considered in this paper, then the right hand side of \eqref{foliation-estimate} is strictly larger than $\dim W$ whenever $\alpha<s$, and hence the estimate \eqref{foliation-estimate} is true and trivial to prove in this case. In specific settings in the Heisenberg group, we can improve on this trivial estimate, using a different method to give nontrivial and asymptotically sharp estimates even in the case that $\alpha \in [\hat{s},s)$. See Section~\ref{Heis foliation} below and \cite{DimDistHeis} for further details.
\end{remark}

\begin{remark}
It seems likely that under additional assumptions on the foliation $(X,W,\pi)$, the conclusion of Theorem \ref{foliation} could be upgraded to
$$\Hdim^{(Q-s)-p\left(1-\frac{s}{\alpha}\right)}_W\left(\{a \in W: \dim f(\pi\inv(a)) \geq \alpha\}\right)=0,$$
as in Theorem \ref{Q-regular}. We leave such a generalization to the interested reader.
\end{remark}

%\begin{remark}\label{net} An $s$-foliation is also an $s'$-foliation for every $s'\geq s$. To prevent this, we could instead require that for every $r>0$, the truncated preimage $\pi \inv(B) \cap K$ contains a maximal $r$-separated set of cardinality $\simeq r^{-s}$. Under mild additional assumptions, this would allow one to prove, for example, that $\dim W \leq Q-s$, at least for certain notions of dimension.
%\end{remark}

\section{Examples of metric foliations}\label{examples section}

In this section we present various examples of metric foliations, and indicate the form that Theorem \ref{foliation} takes in such settings.

\subsection{Euclidean foliations}

As mentioned above, the class of submersions between Riemannian manifolds provides the model example of metric foliations. Note that the submersion assumption is necessary: any smooth surjection $\pi \colon \reals \to \reals$ which is constant on an interval fails to be a $0$-foliation. The canonical metric foliation is given by the orthogonal projection map
$$
P_V \colon \reals^n \to V,
$$
where $V \subeq \reals^n$ is a subspace; this defines an $(n-\dim V)$-foliation. As mentioned in the introduction, the distortion of dimension of leaves of these standard foliations by Sobolev mappings has been extensively studied in \cite{BMT}. In particular, Theorem~\ref{foliation} is a generalization of \cite[Theorem~1.3]{BMT}.

\subsection{Foliations of Sierpi\'nski carpets}

We define a compact subset of $[0,1]^2$ that is homeomorphic to the standard Sierpi\'nski carpet as follows. Let $\mbf{a}=\{a_n\}_{n \in \nats}$ be a sequence of odd integers greater than or equal to three. Divide $[0,1]^2$ into $a_1^2$ squares of side-length $a_1\inv$, and remove the open central square. Repeat this process on each remaining square, removing the central square of side length $(a_1a_2)\inv$, and continue in this fashion \emph{ad infinitum}. If
\begin{equation}\label{converge}
\sum_{n \in \nats} a_n^{-2} <\infty,
\end{equation}
the resulting subset $\mathcal{S}_{\mbf{a}}$ of $[0,1]^2$ is Ahlfors $2$-regular and supports a $p$-Poincar\'e inequality for every $p>1$ \cite{Carpets}.  The restriction $\pi_{\mathcal{S}_{\mbf{a}}}$ of the orthogonal projection $\pi \colon \reals^2 \to \reals \times \{0\}$ to $\mathcal{S}_{\mbf{a}}$ defines a $1$-foliation. Note that in this case the typical leaf of the foliation is a Cantor set of positive length (although some leaves are finite unions of closed intervals). Applying Theorem \ref{foliation} to this example results in the following statement, in which the estimates are the same as in the case of the standard Euclidean projection in $\R^2$.

\begin{corollary}\label{carpet cor}
Suppose that $\mbf{a}$ satisfies \eqref{converge}. Let $p>2$ and $\alpha \in \left(1,\frac{p}{p-1}\right\rbrack$. If $f \colon \mathcal{S}_{\mbf{a}} \to Y$ is a continuous mapping with an upper gradient in $L^p(\mathcal{S}_{\mbf{a}})$. Then
$$\dim \{a \in [0,1] : \dim f(\pi_{\mathcal{S}_{\mbf{a}}}\inv(a))\} \geq \alpha\} \leq 1-p\left(1 - \frac{1}{\alpha}\right).$$
\end{corollary}

Corollary \ref{carpet cor} can also be derived by extending each supercritical Sobolev mapping on $\mathcal{S}_{\mbf{a}}$ to a mapping defined on all of $\reals^2$ of the same regularity, and then applying the results of \cite{BMT}. This is possible as $\mathcal{S}_{\mbf{a}}$ supports a Poincar\'e inequality and has positive two-dimensional measure. Our direct method seems to be simpler.

\subsection{Foliations of the Heisenberg group by left cosets of homogeneous subgroups}\label{Heis foliation}

We describe several natural foliations in the Heisenberg group. These foliations play a starring role in our subsequent paper \cite{DimDistHeis}.

The $n$th Heisenberg group $\Heis^n$, $n \in \nats$, is the unique step two nilpotent stratified Lie group with topological dimension $2n+1$ and one dimensional center. We denote $\Heis^1=\Heis$. Denoting points in $\Heis^n$ by $(x,t) \in \reals^{2n} \times \reals$, the group law is given by
$$(x,t) * (x',t') = \left(x+x',t+t' + 2\omega(x,x')\right),$$
where $\omega(x,x') = \sum_{i=1}^n (x_{n+i}x_i' - x_ix_{n+i}')$ denotes the standard symplectic form on $\reals^{2n}$. The group $\Heis^n$ is equipped with a left-invariant metric $d_\Heis(p,q) = ||p^{-1}*q||_{\Heis}$ via the \emph{Kor\'anyi norm}
$$
||(x,t)||_{\Heis} =( ||x||_{\reals^{2n}}^4 + t^2)^{1/4}.
$$
The metric space $(\Heis^n,d_\Heis)$ is proper and Ahlfors $(2n+2)$-regular when equipped with its Haar measure (which agrees up to constants with both the Lebesgue measure in the underlying Euclidean space $\R^{2n+1}$ and the $(2n+2)$-dimensional Hausdorff measure in the Kor\'anyi metric $d_\Heis$). It is known that $(\Heis^n,d_{\Heis},\Hdim^{2n+2})$ supports a $p$-Poincar\'e inequality for every $1\le p<\infty$; see \cite[Chapter 11]{SobMet} and the references therein.

The Heisenberg group $\Heis^n$ admits a one-parameter family of \emph{intrinsic dilations} $\del_r(x,t)=(rx,r^2t)$, $r>0$. These dilations commute with the group law and are homogeneous of order one with respect to the Kor\'anyi norm, i.e., 
$$\del_r(p)*\del_r(q)=\del_r(p*q) \ \text{and}\ ||\del_r(p)||_\Heis = r ||p||_\Heis.$$

A subgroup of $\Heis^n$ is \emph{homogeneous} if it is invariant under intrinsic dilations. Homogeneous subgroups come in two types. A homogeneous subgroup is called \emph{horizontal} if it is of the form $V \times\{0\}$ for an isotropic subspace $V$ of the symplectic space $\reals^{2n}$. (Recall that $V$ is {\it isotropic} if $\omega|_V=0$.) It is easy to see that every homogeneous subgroup that is not horizontal contains the $t$-axis. The latter subgroups are called \emph{vertical}. Any horizontal subgroup $\V = V \times \{0\}$ defines a semidirect decomposition $\Heis^n = \V^\perp \ltimes \V$ where $\V^\perp = V^\perp \times \reals$ is the \emph{vertical complement} of $\V$; here $V^\perp$ denotes the usual orthogonal complement of $V$ in $\reals^{2n}$.

Since $\omega$ vanishes on isotropic subgroups, the restriction of the Kor\'anyi metric to horizontal homogeneous subgroups coincides with the Euclidean metric. Consequently,
$$
\dim_{\Heis^n} \V = \dim_{\R^{2n+1}} \V = \dim V
$$
for each horizontal homogeneous subgroup; we write $\dim \V$ without any subscript in this case. On the other hand,
$$
\dim_{\Heis^n} \V^\perp = \dim_{\R^{2n+1}} \V^\perp + 1 = \dim V^\perp + 2 = (2n+2)-\dim V.
$$
For example, when $n=1$ we have $\dim \V = 1$ and $\dim_{\Heis} \V^\perp = 3$ for every horizontal line $\V \subset \Heis$.

The semidirect decomposition $\Heis^n = \V^\perp \ltimes \V$ defines maps
$$\pi_{\V} \colon \Heis^n \to \V \ \text{and}\ \pi_{{\V}^\perp} \colon \Heis^n \to ({\V}^\perp,d_\Heis)
$$
by the formulas $\pi_{\V}(p) = p_{\V}$ and $\pi_{\V^\perp}(p) = p_{\V^\perp}$, where $p =
p_{\V^\perp} * p_{\V}$. It is easy to see that $\pi_{\V}$ is Lipschitz on compact sets. However, $\pi_{\V^\perp}$ is not Lipschitz on compact sets (it is at best $\tfrac12$-H\"older). Further information about the metric and measure-theoretic properties of these projection mappings can be found in \cite{balogh-faessler-mattila-tyson2012}.

\begin{proposition}\label{vertical leaves}
The triple $(\Heis^n,\V,\pi_{\V})$ is a $(\dim_{\Heis^n} \V^\perp)$-foliation.
\end{proposition}

\begin{proof}
As already noted, $\pi_\V:\Heis^n\to\V$ is Lipschitz on compact sets. Consider a point $a=(a_V,0) \in \V$ and a radius $r>0$.  Then
$$B_{\V}(a,r) = B_{\Heis}(a,r) \cap \V = \left(B_{\reals^{2n}}(a_V,r) \cap V \right) \times \{0\}.$$
Moreover, as sets,
$$\pi_{\V}\inv(B_{\V}(a,r)) = P_V\inv(B_{\reals^{2n}}(a_V,r) \cap V)  \times \reals,$$
where $P_V \colon \reals^{2n} \to V$ is the standard Euclidean orthogonal projection onto $V$. It follows by volume considerations that for each compact set $K \subeq \Heis^n$, there is a constant $C \geq 1$ depending only on $K$, such that $\pi_{\V}\inv(B_{\V}(a,r)) \cap K$ may be covered by at most $Cr^{-\dim \V^\perp}$ Heisenberg balls of radius~$r$.
\end{proof}

Applying Theorem \ref{foliation} to $\pi_{\V}$ yields the following statement on dimension increase for cosets of a vertical complementary subgroup.

\begin{corollary}\label{Heis lipschitz}
Let $\V$ be a horizontal homogeneous subgroup of $\Heis^n$ and let $Y$ be an arbitrary metric space. Given a continuous mapping $f \colon \Heis^n \to Y$ with upper gradient in $\Lloc^p(\Heis^n)$ for some $p > 2n+2$, and given
$$
\alpha \in \left( \dim_{\Heis^n} \V^\perp, \frac{p\dim_{\Heis^n} \V^\perp}{p-\dim \V}\right\rbrack,
$$
we have the estimate
$$
\dim\{a \in \V: \dim (f(\V^\perp*a)) \geq \alpha\} \leq \dim \V - p\left(1-\frac{\dim_{\Heis^n} \V^\perp}{\alpha}\right).
$$
\end{corollary}

For instance, for any horizontal line $\V$ in $\Heis^1$, any continuous map $f \colon \Heis^1 \to Y$ with upper gradient in $\Lloc^p(\Heis^1)$ for $p > 4$, and any $\alpha \in (3, \frac{3p}{p-1}]$, we have the estimate
\begin{equation}\label{Heis lipschitz estimate}
\dim\{a \in \V: \dim (f(\V^\perp*a)) \geq \alpha \} \leq 1 - p\left(1-\frac{3}{\alpha}\right).
\end{equation}
Note that the upper bound in \eqref{Heis lipschitz estimate} is identical to the one obtained in the classical Euclidean setting for the foliation of $\R^4$ by a one-dimensional family of parallel hyperplanes.

In contrast to $\V$, the restriction of the Kor\'anyi metric to a complementary homogeneous vertical subgroup ${\V}^\perp$ differs dramatically from the restriction of the Euclidean metric. Moreover, as mentioned above, the map $\pi_{{\V}^\perp}$ fails to be Lipschitz on compact sets and can increase the Hausdorff dimension of sets. Thus $\pi_{{\V}^\perp}:\Heis^n \to (\V^\perp,d_{\Heis^n})$ is {\bf not} locally David--Semmes regular. To overcome this difficulty, we alter the choice of metric on $\V^\perp$.

\begin{proposition}\label{horizontal leaves}
The triple $(\Heis^n,(\V^\perp,d_{\reals^{2n+1}}),\pi_{{\V}^\perp})$ is a $(\dim \V+ 1)$-foliation of $\Heis^n$.
\end{proposition}

\begin{proof}
The fact that $\pi_{\V^\perp}$ is Lipschitz on compact sets follows from the fact that the identity map from $\Heis^n$ to $\reals^{2n+1}$ is Lipschitz on compact sets. Moreover, there exists a smooth diffeomorphism $\phi \colon \reals^{2n+1} \to \reals^{2n+1}$ such that  $\pi_{{\V}^\perp} = P_{{\V}^\perp} \circ \phi$, where $P_{{\V}^\perp}$ denotes the Euclidean orthogonal projection onto ${\V}^\perp$. Hence, given $a \in {{\V}^\perp}$, $r>0$, and a compact set $K \subeq \Heis^n$, there is a constant $C \geq 1$, depending only on $K$, such that the set $\pi_{{\V}^\perp}\inv(B_{\reals^{2n+1}}(a,r)) \cap K$ can be covered by at most  $C r^{-\dim \V}$ Euclidean balls of radius $r$. It follows by an application of the Ball-Box Theorem that there is another constant $C' \geq 1$, depending only on $K$, such that $\pi_{{\V}^\perp}\inv(B_{\reals^{2n+1}}(a,r)) \cap K$ can be covered by $C'r^{-(\dim \V+1)}$ balls in the Kor\'anyi metric $d_{\Heis^n}$.
\end{proof}

Observe that the Hausdorff dimension of each leaf of the foliation defined by $\pi_{{\V}^\perp}$ is equal to $\dim \V$ and not $(\dim \V+1)$. As discussed in Remark \ref{limitations}, in this situation we do not expect a particularly good estimate to arise from Theorem \ref{foliation}. Nevertheless we record the following corollary.

\begin{corollary}\label{Heis not Lipschitz}
Let $\V$ be a horizontal homogeneous subgroup of $\Heis^n$ and let $Y$ be an arbitrary metric space. Given a continuous mapping $f \colon \Heis^n \to Y$ with upper gradient in $\Lloc^p(\Heis^n)$ for some $p > 2n+2$, and given
$$
\alpha \in \left( \dim \V+1, \frac{p(\dim \V +1)}{p-\dim_{\R^{2n+1}} \V^\perp} \right\rbrack,
$$
we have the estimate
\begin{equation}\label{Heis not Lipschitz estimate}
\dim_{\reals^{2n+1}}\{a \in \V^\perp: \dim(f(a*\V)) \geq \alpha \} \leq \dim_{\R^{2n+1}} \V^\perp - p\left(1-\frac{\dim \V +1}{\alpha}\right).
\end{equation}
\end{corollary}

Above we use the notation $\dim_{\reals^{2n+1}}$ to emphasize that we consider the Hausdorff dimension of the set equipped with the Euclidean metric.

By an application of the Dimension Comparison Theorem \cite{BTW} we deduce from \eqref{Heis not Lipschitz estimate} the following estimate
\begin{equation}\label{Heis not Lipschitz estimate 2}
\dim_{\Heis^n}\{a \in \V^\perp: \dim (f(a*\V)) \geq \alpha \} \leq \dim_{\Heis^n} \V^\perp - p\left(1-\frac{\dim \V +1}{\alpha}\right).
\end{equation}
For example, when $n=1$ estimate \eqref{Heis not Lipschitz estimate 2} reads
%$$
%\dim_{\reals^3}\{a \in \V^\perp: \Hdim^\alpha(f(a*\V)) >0\} \leq 2 - p\left(1-\frac{2}{\alpha}\right).
%$$
%and
$$
\dim_{\Heis^1}\{a \in \V^\perp: \dim(f(a*\V)) \geq \alpha \} \leq 3 - p\left(1-\frac{2}{\alpha}\right).
$$
%respectively.

We now reiterate the ways in which the estimate in Corollary \ref{Heis not Lipschitz} is deficient. As mentioned above, for each $a \in \V^\perp$, the Heisenberg metric on the leaf $a*\V$ coincides with the restriction of the Euclidean metric, and the resulting space is Ahlfors $(\dim\V)$-regular. Hence Theorem~\ref{Q-regular} implies that given $f$ as in Corollary \ref{Heis not Lipschitz},
\begin{equation}\label{p-3}
\left\{a \in \V^{\perp} : \Hdim^{\frac{p\dim \V}{(p-\dim_{\Heis^n}\V^\perp)}}(f(a*\V)) >0\right\} = \emptyset.
\end{equation}
However, applying Corollary \ref{Heis not Lipschitz} with $\alpha  = \frac{p\dim \V}{(p-\dim_{\Heis^n}\V^\perp)}$ yields only
$$
\dim_{\reals^{2n+1}}\left(\left\{a \in \V^{\perp} : \Hdim^{\frac{p\dim \V}{(p-\dim\V^\perp)}}(f(a*\V)) >0\right\}\right) < \frac{p-(2n+2)}{\dim\V},
$$
and the quantity on the right hand side is strictly greater than zero. Moreover, Theorem \ref{foliation} can provide no information of about the frequency with which a supercritical Sobolev mapping maps leaves onto sets of dimension at least $\alpha$ when $\alpha \in [\dim \V,\dim \V+1]$.

These deficiencies are addressed in \cite{DimDistHeis}, which presents a comprehensive study of dimension increase properties of Sobolev mappings of the Heisenberg group $\Heis^n$ on elements of such foliations.

\subsection{Foliations of $\Heis^1$ by right cosets of horizontal lines}

As a final example we specialize to the first Heisenberg group $\Heis$ and consider the foliation by {\it right} cosets of a horizontal line. As we shall see, this foliation is well behaved with respect to the underlying non-Riemannian geometry of both the Heisenberg group and the parameterizing space, and leads to good estimates for dimension increase arising from our main theorems.

We recall that the sub-Riemannian geometry of $\Heis$ is defined via the {\it horizontal distribution} $H\Heis$, the unique left-invariant rank two subbundle of the tangent bundle $T\Heis$ for which $H_e\Heis = \spa\{\tfrac{\partial}{\partial {x_1}},\tfrac{\partial}{\partial {x_2}}\}$, where $e=(0,0)$ denotes the identity element of $\Heis$.
We denote by $X_1$ and $X_2$ the left-invariant vector fields on $\Heis$ whose values at $e$ agree with $\tfrac{\partial}{\partial{x_1}}$ and $\tfrac{\partial}{\partial{x_2}}$ respectively; then $H_p\Heis = \spa\{(X_1)_p,(X_2)_p\}$.

A smooth curve $\gamma:[a,b]\to\Heis$ is {\it horizontal} if $\gamma'(s) \in H_{\gamma(s)}\Heis$ for all $s$. We define the {\it length} of $\gamma$ by declaring $X_1$ and $X_2$ to be an orthonormal frame in $H\Heis$, and we introduce the {\it Carnot-Carath\'eodory (CC) metric} $d_{cc}$ on $\Heis$ as the geodesic metric obtained by infimizing the lengths of horizontal curves joining two given points. It is well known that the Carnot-Carath\'eodory metric $d_{cc}$ and the Kor\'anyi metric $d_\Heis$ are comparable; this fact easily follows from the observation that both metrics are left invariant and homogeneous of order one with respect to intrinsic dilations.

Let us fix a horizontal line $\V$ in $\Heis$ and consider the semidirect decomposition $\Heis = \V \rtimes \V^\perp$; note that the normal subgroup $\V^\perp$ now appears on the right. The right cosets $\V*a$, $a \in \V^\perp$, are typically not horizontal curves (only in the case when $a$ lies in the center of $\Heis$, i.e., the $t$-axis, is $\V*a$ a horizontal line). We define a map
$$
\pi^R_{{\V}^\perp} \colon \Heis \to {\V}^\perp
$$
by the formula $\pi^R_{\V^\perp}(p) = p^R_{\V^\perp}$, where $p = p_\V * p^R_{\V^\perp}$. Identifying $\V^\perp$ with the collection $\cX = \{\V*a:a \in \V\}$ of right cosets of $\V$, the map $\pi^R_{{\V}^\perp}$ coincides with the quotient map $p\mapsto[p]$, where $[p]$ denotes the unique right coset $\V*a$ containing $p$.

The CC metric on $\Heis$ induces a well defined metric on $\cX$ by the formula $\dist_{cc}(\V*a,\V*a')$. Moreover, the left invariance of the CC metric implies that right cosets are {\it CC parallel}: $\dist_{cc}(\V*a,\V*a') = \dist_{cc}(x*a,\V*a')$ for any $x \in \V$.

The {\it Grushin plane} $G$ is the two-dimensional sub-Riemannian structure on $\reals^2$ defined by the horizontal distribution $HG$ given by $H_{(u,v)}G = \R^2$ if $u\ne 0$ and $H_{(0,v)}G = \R\times\{0\}$. A curve $\gamma$ in $G$ is {\it horizontal} if its tangent vectors lie everywhere in the horizontal distribution, i.e., if the second component of $\gamma'(s)_2$ is zero whenever the first component of $\gamma(s)$ is zero.

The Carnot-Carath\'eodory metric $d_{cc}$ on $G$ is defined as for the Heisenberg group:
$$
d_{cc}((u_1,v_1),(u_2,v_2)) = \inf_\gamma \int_a^b \sqrt{(u'(s))^2 + \frac{(v'(s))^2}{(u(s))^2}} \, ds,
$$
where the infimum is taken over all horizontal curves $\gamma=(u,v) :[a,b]\to G$ that connect $(u_1,v_1)$ to $(u_2,v_2)$. Formally, this corresponds to the choice of the orthonormal frame $\{U,V\} = \{\tfrac{\partial}{\partial u},u\tfrac{\partial}{\partial v}\}$ for $HG$, note however that $V_{(0,v)}=0$ for all $v\in\R$, so this is not a genuine frame.

The following fact is well known, see e.g.\ Arcozzi--Baldi \cite[Theorem 1]{ArcozziBaldi2008}. It is a specific instance of the celebrated Rothschild--Stein lifting theorem for families of H\"ormander vector fields \cite{RothschildStein1976}.

\begin{theorem}\label{X and G}
The space $\cX$ is isometric to the Grushin plane $(G,d_{cc})$.
\end{theorem}

Identifying $\cX$ with $G$ and considering $\pi^R=\pi^R_{{\V}^\perp}$ as a map from $\Heis$ to $G$, we will show

\begin{proposition}\label{Heis-Grushin-foliation}
The triple $(\Heis,G,\pi^R)$ is a $2$-foliation.
\end{proposition}

Before giving the proof of Proposition \ref{Heis-Grushin-foliation} we indicate the estimates for dimension increase which follow from that proposition in combination with Theorem \ref{foliation}.

\begin{corollary}\label{Heis-Grushin-corollary}
Let $Y$ be any metric space. For $p>4$, if $f \colon \Heis \to Y$ is a continuous mapping that has an upper gradient in $\Lloc^p(\Heis^1)$, then
\begin{equation}\label{Heis-Grushin-estimate}
\dim \{w \in G: \dim (f((\pi^R)\inv(w)))\geq \alpha \} \leq 2-p\left(1-\frac{2}{\alpha}\right)
\end{equation}
for each $\alpha \in \left( 2,\frac{2p}{p-2}\right\rbrack$.
\end{corollary}
Note that the upper bound in \eqref{Heis-Grushin-estimate} is identical to the one obtained in the classical Euclidean setting for the foliation of $\R^4$ by a two-dimensional family of parallel $2$-planes. This is consistent with the fact that the Grushin plane has Hausdorff dimension two, the Heisenberg group has Hausdorff dimension four and the typical leaf in the foliation has Hausdorff dimension two.

In the proof of Proposition \ref{Heis-Grushin-foliation}, we will make use of the following explicit two-sided estimate for the CC metric in the Grushin plane $G$: there exists an absolute constant $C_1 \ge 1$ so that
\begin{equation}\label{Grushin-two-sided}
\frac1{C_1} \le \frac{\max \{ |u_1-u_2| , \min \{ \sqrt{|v_1-v_2|} , \frac{|v_1-v_2|}{\max\{|u_1|,|u_2|\}} \} \}}{d_{cc}((u_1,v_1),(u_2,v_2))} \le C_1
\end{equation}
whenever $(u_1,v_1) \ne (u_2,v_2)$. See, for instance, Bella{\"\i}che \cite{Bell1996} or Seo \cite{Seo2011}. Here we interpret the quantity
$$
\frac{|v_1-v_2|}{\max\{|u_1|,|u_2|\}}
$$
to be $+\infty$ if $u_1=u_2=0$.

\begin{proof}[Proof of Proposition \ref{Heis-Grushin-foliation}]
Theorem \ref{X and G} shows that the projection $\pi^R:\Heis\to G$ is a $1$-Lipschitz mapping from $(\Heis,d_{cc})$ to $(G,d_{cc})$. It follows that $\pi^R$ is Lipschitz from $(\Heis,d_\Heis)$ to $(G,d_{cc})$.

Now suppose that $K$ is a compact subset of $\Heis$ and that $B_0=B_{cc}(w_0,r)$ is a ball in $G$ with radius $r<1$. We claim that $(\pi^R)\inv(B_0)\cap K$ can be covered by $C/r^2$ balls in the Kor\'anyi metric $d_\Heis$ of radius $Cr$. Here $C>0$ denotes a quantity, possibly varying at each instance, depending only on $K$. By the Ball-Box Theorem (cf.\ the proof of Proposition \ref{horizontal leaves} above), it suffices to show that $(\pi^R)\inv(B_0)\cap K$ can be covered by $C/r$ Euclidean balls of radius $Cr$. We will prove the latter statement by volume considerations; it is enough to prove that the Lebesgue volume of $(\pi^R)\inv(B_0)\cap K$ is less than or equal to $Cr^2$.

In order to compute the volume, we need good control on the Lebesgue area of the Grushin CC ball $B_{cc}(w_0,r)$. Since vertical translation is an isometry of $G$, it suffices to consider balls $B_{cc}(w_0,r)$ centered on the $v$-axis, i.e., $w_0 = (u_0,v_0)$ with $v_0=0$. As $\pi^R$ is Lipschitz on compact sets, we may assume that $|u_0|<C$. Denote by $A(w_0,r)$ the Lebesgue area of $B_{cc}(w_0,r)$. For $(u,v)$ in this ball, \eqref{Grushin-two-sided} implies that 
\begin{align*}  |u-u_0|\cdot |v| & \leq Cr\max\left\{r^2,r(\max \{|u|,|u_0|\})\right\} \\
&\leq  Cr\left(r^2+r(|u|+|u_0|)\right) \\
&\leq Cr\left(r^2 + r(r+2|u_0|)\right) \\
&\leq C(r^2+r^3) \\
&\leq Cr^2.
\end{align*}

%If $w_0=(0,0)$ and $w = (u,v) \in B_{cc}(w_0,r)$ then \eqref{Grushin-two-sided} implies that
%$$
%\max\left\{ |u|, \min \left\{ \sqrt{|v|} , \frac{|v|}{|u|} \right\} \right\} \le r' := C_1 r;
%$$
%hence $|u|\le r'$ and $|v| \le (r')^2$. Consequently in this case $A(w_0,r) \le 4(r')^3 \le Cr^2$ (since $r'=C_1r$ and $r\le 1$).
%
%If $w_0 = (u_0,0)$ for some $u_0 \ne 0$ and $w = (u,v) \in B_{cc}(w_0,r)$, then \eqref{Grushin-two-sided} implies that
%$$
%\max\left\{ |u-u_0|, \min \left\{ \sqrt{|v|} , \frac{|v|}{\max\{|u|,|u_0|\}} \right\} \right\} \le r' := C_1 r;
%$$
%hence $|u-u_0| \le r'$ and either $|v| \le (r')^2$ or $|v| \le r' \max \{|u|,|u_0|\}$.
%We consider two cases. If $r'<2|u_0|$ then $A(w_0,r) \le 4(r')^2|u_0|+(r')^3 \le C r^2$. If $r'>2|u_0|$ then $A(w_0,r) \le 4(r')^2|u_0|+(r')^3+(r')(r'-2|u_0|)^2 \le C r^2$ as before.
We conclude that $A(w_0,r) \le C r^2$. By an argument similar to that in the proof of Proposition~\ref{horizontal leaves},
it follows that the Lebesgue volume of $(\pi^R)\inv(B_0)\cap K$ is less than or equal to $Cr^2$, as desired. This completes the proof of Proposition \ref{Heis-Grushin-foliation}.
\end{proof}

%\subsection{Foliations of product Carnot groups}

%A \emph{Carnot group} is a connected and simply connected Lie group $\mbf{G}$ whose Lie algebra $\mathfrak{g}$ has a nilpotent %stratification; i.e., $\mathfrak{g} = V_1 \bigoplus \hdots \bigoplus V_r$, where $\lbrack V_1,V_j \rbrack = V_{j+1}$ for $j=1,\hdots %,r-1$ and $\lbrack V_1,V_r \rbrack = \{0\}.$ The natural number $r$ is the \emph{step} of the group $\mbf{G}$. A Carnot group may be %equipped with the \emph{Carnot-Carath\'eodory} metric, with respect to which the Haar measure is Ahlfors regular in the %\emph{homogeneous dimension} $Q_{\mbf{G}} = \sum_{i=1}^r i \dim V_i$. We refer to \cite{Seen} and \cite{Hardy} for further details %about Carnot groups.

%In the case of a Carnot group $\mbf{G} = \mbf{G}_1 \times \mbf{G}_2$ of product type, the natural projection $\pi \colon \mbf{G} \to %\mbf{G}_1$ is a Lipschitz surjection and defines a $Q_{\mbf{G}_2}$-metric foliation, the leaves of which are copies of $\mbf{G}_2$. %Note that the parameterizing space of this foliation is not Euclidean, and so the methods of \cite{DimDistHeis} based on the %Radon--Nikodym theorem do not apply.

\section{Open problems and questions}\label{questions section}

In this final section we collect several open problems and questions motivated by the present work.

Note that every $s$-foliation is also an $s'$-foliation for each $s'\ge s$. Let us call the {\it minimal foliation exponent} for a metric foliation $\pi:X \to W$ the infimum of the values $s$ for which $\pi$ is an $s$-foliation.  The following question is inspired by Theorem \ref{universal-sharp}.

\begin{question}\label{q}
Let $X$ be a locally Ahlfors $Q$-regular metric space and assume that $\pi \colon X \to W$ is an $s$-foliation, where $s$ is the minimal foliation exponent. Assume also that $\hat{s}=\dim \pi \inv (a)$ is independent of $a \in W$. Let $p>Q$, let $\alpha \in [\hat{s}, \frac{p\hat{s}}{p-Q+s}]$, and let $N$ be an integer greater than $\alpha$. If $W$ possesses a subset $E$ that is evenly coverable and has dimension
$$\beta = (Q-s) - p\left(1-\frac{\hat{s}}{\alpha}\right),$$
then does there exists a continuous mapping $f \colon X \to \reals^N$ with an upper gradient $\Lloc^p(X)$ such that
$$\dim(f(\pi\inv(a))) \geq \alpha$$
for all $a \in E$?
\end{question}
We believe that the answer to this question is yes and give a similar construction in \cite{DimDistHeis}.

As regards the positive result in Theorem \ref{foliation}, the discussion in Remark \ref{limitations} shows that the estimates given in that theorem are not natural in case
$$
\hat{s} := \sup \{ \dim \pi\inv(a) : a \in W \} < s.
$$
In view of Question \ref{q}, one might instead wish for the estimate
\begin{equation}\label{s-hat-estimate}
\dim \{a \in W: \Hdim^\alpha(f(\pi\inv(a)))>0\} \leq (Q-s)-p\left(1-\frac{\hat{s}}{\alpha}\right)
\end{equation}
as the conclusion of Theorem \ref{foliation}. We are not able to prove the estimate \eqref{s-hat-estimate} for the standard foliation of the Heisenberg group $\Heis$ by horizontal lines, although in our forthcoming work \cite{DimDistHeis} a very similar estimate is achieved for when $\alpha$ is close to $\hat{s}$.  One could also inquire if the yet weaker estimate
\begin{equation}\label{weak s-hat-estimate}
\dim \{a \in W: \Hdim^\alpha(f(\pi\inv(a)))>0\} \leq (Q-\hat{s})-p\left(1-\frac{\hat{s}}{\alpha}\right)
\end{equation}
holds; in \cite{DimDistHeis} we achieve this estimate for the foliation of the Heisenberg group by horizontal lines when $\alpha$ is close to the universal bound.
\begin{question}
In which situations are the estimates \eqref{s-hat-estimate} and \eqref{weak s-hat-estimate} valid?
\end{question}

The proof of Theorem \ref{universal-sharp} suggests a general meta-theorem deriving prevalence theorems for dimension increase from specific examples.

Let $(X,d)$ be a metric space and let real numbers $s$ and $\alpha$ and an integer $N$ satisfy $0<s<\dim X$ and $s\le\alpha<N$. Let us say that a normed linear class $\cF$ of mappings from $X$ to $\R^N$ is {\it $(s,\alpha)$-prevalence forcing} in case the following condition holds: if to each compact set $E\subset X$ with $\cH^s(E)>0$ there corresponds a mapping $f \in \cF$ such that $\dim f(E) \ge \alpha$, then the set of all mappings in $\cF$ with that property is prevalent.

Theorem \ref{universal-sharp} asserts that the Sobolev--Newtonian class ${\rm N}^{1,p}(X;\R^N)$ on a locally Ahlfors $Q$-regular metric measure space $(X,d,\mu)$ is $(s,\tfrac{ps}{p-Q+s})$-prevalence forcing for each $0<s<Q<p$.

\begin{question}
Let $(X,d,\mu)$ be a metric measure space and let $N\in\nats$. For given $\alpha$ and $s$ satisfying the preceding constraints, which normed linear classes $\cF$ of mappings from $X$ to $\R^N$ are $(s,\alpha)$-prevalence forcing?
\end{question}

Finally, motivated by the discussion in Section \ref{Heis foliation}, we pose

\begin{question}
Is there a $1$-foliation of the Heisenberg group?
\end{question}

\bibliographystyle{acm}
\bibliography{DimDistMetric}
\end{document}